\documentclass[10pt]{amsart} 

\usepackage[margin=3cm]{geometry}

\usepackage[utf8]{inputenc} 
\usepackage[T1]{fontenc}

\usepackage{geometry} 
\usepackage{graphicx} 
\usepackage{amsmath}
\usepackage{amsthm}
\usepackage{comment}
\usepackage{amssymb}
\usepackage{color}
\newcommand{\abs}[1]{\left\vert#1\right\vert}
\newcommand{\norm}[1]{\left\lVert#1\right\rVert}

\newcommand{\E}{\mathbb{E}}

\newcommand{\G}{\mathbb{G}}
\newcommand{\R}{\mathbb{R}}

\newcommand{\cadlag}{c\`adl\`ag }

\newcommand{\cP}{\mathcal{P}}

\newtheorem{theorem}{Theorem}[section]
\newtheorem{lemma}[theorem]{Lemma}
\newtheorem{proposition}[theorem]{Proposition}

\theoremstyle{definition}
\newtheorem{remark}{Remark}
\newtheorem*{notation}{Notation}
\newtheorem{definition}{Definition}

\usepackage{bm}

\newtheorem{assump}{Assumption}

\usepackage[backend=bibtex,style=numeric,citestyle=numeric]{biblatex}
\bibliography{biblioapp}

\begin{document}
\title[Nash equilibrium in games with mean-field interaction and controlled jumps]{$\varepsilon$-Nash equilibrium in stochastic differential games \\ with mean-field interaction and controlled jumps} 

\author{Chiara Benazzoli}
\address{Department of Mathematics, University of Trento, Italy}
\email{c.benazzoli@unitn.it}

\author{Luciano Campi}
\address{Department of Statistics, London School of Economics, UK}
\email{l.campi@lse.ac.uk}

\author{Luca Di Persio}
\address{Department of Computer Science, University of Verona, Italy}
\email{luca.dipersio@univr.it}

\date{\today}

\maketitle

\begin{abstract}
We consider a symmetric $n$-player nonzero-sum stochastic differential game with controlled jumps and mean-field type interaction among the players. Each player minimizes some expected cost by affecting the drift as well as the jump part of their own private state process. We consider the corresponding limiting mean-field game and, under the assumption that the latter admits a regular Markovian solution, we prove that an approximate Nash equilibrium for the $n$-player game can be constructed for $n$ large enough, and provide the rate of convergence. This extends to a class of games with controlled jumps classical results in mean-field game literature. This paper complements our previous work \cite{BCDP}, where in particular the existence of a mean-field game solution was investigated. \medskip
\end{abstract}

\section{Introduction}
In this paper we consider a symmetric nonzero-sum stochastic differential game with controlled jumps, where the interaction among the players is of mean-field type. Our goal is to extend to this setting classical results on constructing approximate Nash equilibrium, for a large number of players, using the mean-field game approach. 

Mean-field games (MFGs, henceforth) are optimization problems that were simultaneously introduced by Lasry and Lions in \cite{lasry2006jeux1,lasry2006jeux2,lasry2007mean} and by Huang and co-authors in \cite{huang2006large}. They have to be understood as an approximation of large population symmetric stochastic differential games, whose players interact via the empirical law of their private states. According to such an approach, when the number of players, $n$, is large enough, a solution of the limit MFG can be used to provide a nearly Nash equilibria for the corresponding
$n$-player games; see, e.g., \cite{huang2006large}, \cite{kolokoltsov2011mean}, \cite{carmona2013mean}, \cite{carmona2013probabilistic}, \cite{carmona2015probabilistic}  as well as the recent book \cite{carmona2016lectures}. This approximation result is even more relevant as computing Nash equilibria in $n$-player games with $n$ very large is usually not feasible even numerically due to the curse of dimensionality. Moreover, MFGs represent a very flexible framework for applications in various areas including but not limited to finance, economics and crowd dynamics (see \cite{gueant2011mean,carmona2016lectures} for a good sample of applications), which partly explain the increasing literature on the subject.\medskip

The $n$-player game we consider can be shortly described as follows. Consider the following dynamics for the private state, $X^i$, of player $i=1,\ldots,n$:
\begin{equation}\label{X^{i,n}} dX^{i,n} _t =     b(t,X^{i,n} _t, \mu^n _t ) dt + \sigma(t,X^{i,n} _t) dW^i _t + \beta(\mu^n _{t-},\gamma^i_t) d\widetilde N^i _t,\quad t\in[0,T],\end{equation}
where $\gamma^i$ is the control of player $i$, $\mu^n _t = (1/n)\sum_{j=1}^n \delta_{X_t ^{j,n}}$ is the empirical distribution of the vector $(X_t ^{1,n},\ldots, X_t ^{n,n})$ of the private states of all players, $(W^1,\ldots,W^n)$ is an $n$-dimensional Brownian motion and $(\widetilde N^1,\ldots, \widetilde N^n)$ is an $n$-dimensional (compensated) jump process. The goal of each player $i$ is minimizing some objective functional, given by
\begin{equation}
 \E\left[ \int_0 ^T f(t,X^{i,n} _t, \mu^n _t ,\gamma^i _t) dt +g(X^{i,n} _T, \mu^n _T)\right],
\end{equation}
over her/his controls, for some running cost $f$ and some final cost $g$. In a nutshell we are dealing with a symmetric stochastic differential game, where the agents interact through the empirical distribution of their private states, entering in both the drift and the jump component. 

According to mean-field game theory, we expect that as the number of players gets larger and larger, the $n$-player game just described tends in some sense to the following MFG with controlled jumps, which has been the object of our previous paper  \cite{BCDP}, where we focused on the existence of solutions for the limit MFG and we gave sufficient conditions granting the existence of a Markovian solution. Moreover, in \cite{BCDP} we also provided an application to an illiquid interbank model where everything can be computed explicitly.

In the limiting MFG, $Y=Y^\gamma$ denotes a state variable following the dynamics
\begin{equation}\label{Xrelax} dY_t =  b(t,Y_t,\mu_t ) dt + \sigma(t,Y_t) dW_t +  \beta(\mu_{t-},\gamma_t ) d\widetilde N_t,\quad t\in[0,T],\end{equation}
where $(\mu_t )_{t\in [0,T]}$ is a deterministic right-continuous with left limit (henceforth c\`adl\`ag) measure flow, i.e. any $\mu_t$ belongs to the space of all probability measures on the real line $\mathcal{P}(\R)$ equipped with the topology of the weak convergence, $\gamma_t$ represents a control process, taking values in a fixed action space $A$, $W$ is a standard Brownian motion and $ N$ is some (compensated) jump process. Moreover, we assume that $W$ and $N$ are independent. The aim is to find a control $\hat \gamma$ solving the following minimization problem
\begin{equation}
\inf_{\gamma} \E\left[ \int_0 ^T f(t,Y_t,\mu_t ,\gamma_t) dt+g(Y_T,\mu_T )\right],
\end{equation}
over all control processes $\gamma$ as above and such that the so-called mean-field condition is fulfilled: the measure flow $\mu$ has to be equal to the law of the optimally controlled dynamic $\widehat Y = Y^{\hat \gamma}$ at each time, that is $\mu_t = \mathcal{L}(\widehat Y_t)$ 
for all $t\in [0,T]$.

Our main contribution is that any solution to the latter, provided it is Markovian and Lipschitz continuous in the state variable, provides a good approximation of some Nash equilibrium in the $n$-player game. More precisely, let $\hat \gamma(t,Y_{t-})$ be a Markovian MFG solution with $\hat \gamma (t,x)$ Lipschitz continuous in $x$, hence the strategy profile $(\hat \gamma (t,X^{1,n}_{t-}),\ldots, \hat \gamma (t,X^{n,n}_{t-}))$, for $t\in [0,T]$, is an $\varepsilon_n$-Nash equilibrium in the $n$-player game, where $\varepsilon_n \to 0$ as $n \to \infty$. This result extends to our jump setting classical results that have been proven for continuous paths state variables as in, e.g., \cite{carmona2013mean,carmona2013probabilistic,carmona2015probabilistic}. Notice that our model is one-dimensional only for the sake of simplicity, extending our results to a multi-dimensional state variable $X$ is straightforward. 
 
We conclude the introduction with a brief overview on MFG for jump processes. While the uncontrolled counter-part of MFG, that is particle systems and propagation of chaos for jump processes, has been thoroughly studied in the probabilistic literature (see, e.g., \cite{Gra,jourdain2007nonlinear} and the very recent preprint \cite{andreis2016MKVjumps}), MFGs with jumps have not attracted much attention so far. Indeed, most of the existing literature focuses on non-linear dynamics with continuous paths, with the exception of few papers such as \cite{gomes, hafayed2014mean, kolokoltsov2011mean}, and the more recent \cite{cecchin2017probabilistic}. The paper \cite{gomes} studies MFGs in continuous-time with finitely many states, while \cite{hafayed2014mean} deals with stochastic control of McKean-Vlasov type (see \cite{carmona2013control} for a comparison between MFG and McKean-Vlasov control), and \cite{kolokoltsov2011mean} uses methods based on potential theory and nonlinear Markov processes. The latter \cite{cecchin2017probabilistic} presents a new probabilistic approach based on jump processes used to get results on MFG with finitely many states. \medskip

The paper is organized as follows: in Section \ref{N-game} we describe the $n$-player game, the corresponding MFG and give all the relevant assumptions on the initial data and the coefficients. Section \ref{nash} contains the statement and the proof of the main theorem of this paper, establishing that under suitable assumptions a Markovian MFG solution can be used to construct an approximate Nash equilibrium for the $n$-player game with $n$ sufficiently large.

\section{A symmetric $n$-player game with interaction of mean-field type}\label{N-game}

 In this section we give a precise description of the $n$-player game we are interested in, together with the corresponding MFG. Moreover we set the main assumptions and prove some preliminary a-priori estimates on the state variables.

\subsection{The $n$-player game $G_n$ with mean-field interaction.}
Let $(\Omega, \mathcal{F}, (\mathcal{F}_t)_{t\in[0,T]}, P)$ be a filtered probability space satisfying the usual conditions and supporting $n$ independent Brownian motions $W^i$ and $n$ independent Poisson processes $N^i$ with a time-dependent intensity $\lambda(t)$, for $i=1,\dots,n$. Let $X^{i,n}$ be the unique strong solutions to the following SDEs 
\begin{equation}\label{eq:XiN}
\begin{aligned}
& dX^{i,n}_t(\gamma)=b(t,X^{i,n}_t,\mu^n_t) dt +\sigma(t,X^{i,n}_t) dW^i_t+\beta(\mu^n_{t-},\gamma^i_{t}) d\widetilde N^i_t ,\\
& X^{i,n}_0=\xi^i ,
\end{aligned}
\end{equation}
where the initial values $\xi^i$, $i=1,\dots, n$, are independent real-valued random variables, all distributed according to the same distribution $\chi$. Here, $\mu^n$ denotes the empirical distribution of the system $X^n=(X^{1,n},\dots, X^{n,n})$ at time $t$, which is defined as
\begin{equation}\label{def:empirical}
\mu^n_t=\mu^n_t(\gamma)=\frac{1}{n}\sum_{i=1}^n \delta_{X^{i,n}_t(\gamma)}\,,
\end{equation}
where $\delta_{\cdot}$ is the Dirac measure.

Each player $i$ chooses his/her control $\gamma^i$, also called strategy of Player $i$, which takes values in the control space $A$. $A$ is assumed to be a compact, convex set in $\R$. To be admissible, a control $\gamma^i$ has to be a predictable process. The set of all the admissible strategies will be denoted by $\G$ and ${A}_\infty$ will denote $\sup_{a\in A} \abs{a}$, which is finite by compactness of $A$.
An admissible strategy profile $\gamma$ for the game $G_n$, also called simply an admissible strategy, is an $n$-tuple $(\gamma^1,\dots,\gamma^n)$ of admissible controls $\gamma^i\in\G$ for all $i=1,\dots,n$, i.e. $\gamma\in\G^n$, where $\gamma^i$ represents the action chosen by Player $i$. Assumptions guaranteeing the existence of a unique strong solution to the SDEs above will be given later.

For Player $i$, the expected outcome of the game $G_n$ according to the strategy profile $\gamma=(\gamma^1,\dots,\gamma^n)$ is defined by
\begin{equation}\label{def:JiN}
J^{i,n}(\gamma)=\E\left[\int_0^T f(t,X^{i,n}_t(\gamma),\mu^n_t(\gamma),\gamma^i_t) dt+g(X^{i,n}_T(\gamma), \mu^n_T(\gamma))\right]\,.
\end{equation}
We assume that the functional $J^{i,n}$ represents a \emph{cost} for the agent and that all the players are rational. Therefore, the aim of each player in the game is to minimize $J^{i,n}(\gamma)$ over his/her admissible strategies $\G$.

We write $X^{i,n}(\gamma)$ and $J^{i,n}(\gamma)$ to stress that the dynamics of the state and the expected cost of game $G_n$ of Player $i$ depend not only on his/her control $\gamma^i$ but also on the decision rule of the other players. The interaction between the players is of mean-field type: the dynamics of each player private state and his/her costs depend on the other players' states only through their distribution. Therefore the cost functions and the private state dynamics for each player $i$ are invariant under a permutation of the other players' identities, and this provides symmetry to the game $G_n$.

\begin{notation}
Given an admissible  strategy profile $\gamma=(\gamma^1,\dots,\gamma^n)\in\G^n$ and an admissible strategy $\eta\in\G$, $(\eta,\gamma_{-i})$ denotes a further admissible strategy where Player $i$ deviates from $\gamma$ by playing $\eta$, whereas all the other players continue playing $\gamma^j$, $j\neq i$, i.e.
\[
(\eta,\gamma_{-i})=(\gamma^1,\dots,\gamma^{i-1},\eta,\gamma^{i+1},\dots,\gamma^n)\,.
\]
\end{notation}

Our aim is to find an approximate Nash equilibrium for the $n$-player game $G_n$.
\begin{definition}
For a given $\varepsilon\ge0$, an admissible strategy profile $\gamma=(\gamma^1,\dots,\gamma^n)\in \G^n$ is an  $\varepsilon$-Nash equilibrium of the $n$-player game $G_n$ if for each $i=1,\dots, n$ and for any admissible strategy $\eta\in \G$ the following inequality is satisfied
\begin{equation}\label{eq:eNe}
J^{i,n}(\eta, \gamma_{-i})\ge J^{i,n}(\gamma)-\varepsilon\,.
\end{equation}
A strategy $\gamma$ is a Nash equilibrium of the $n$-player game $G_n$ if it is a $0$-Nash equilibrium, that is an $\varepsilon$-Nash equilibrium with $\varepsilon=0$.
\end{definition}
In other words, a strategy profile $(\gamma^1,\dots,\gamma^n)$ is an $\varepsilon$-Nash equilibrium if for each player in the game an unilateral change of his/her strategy when the others remain unchanged provides a maximum saving of $\varepsilon$.

We now introduce a mean-field game, which represents the previous game $G_n$ when the number of players $n$ grows to infinity.\\

\subsection{The associated mean-field game $G_\infty$ and main assumptions.}
Let $(\Omega, \mathcal{F},(\mathcal{F}_t)_{t\in[0,T]}, P)$ be a filtered probability space satisfying the usual conditions and supporting a Brownian motion $W$ and a Poisson process $N$ with a time-dependent intensity function $\lambda(t)$. Let $Y$ be the unique strong solution to
\begin{equation}\label{eq:Y}
\begin{aligned}
& dY_t(\gamma)=b(t,Y_t,\mu_t) dt+\sigma(t,Y_t) dW_t+\beta(\mu_{t-},\gamma_{t}) d\widetilde N_t\\
& Y_0=\xi\sim\chi
\end{aligned}
\end{equation}
where $\mu$ is a \cadlag flow of probability measures, $\mu\colon[0,T]\to \cP(\R)$, with $\mu(0-)=\delta_0$. Assumptions granting the existence of a unique strong solution to the SDE above will be given shortly.
The expected outcome of the game for playing strategy $\gamma$ is defined by
\begin{equation}\label{def:Jinf}
J(\gamma)=\E\left[\int_0^T f(t,Y_t(\gamma),\mu_t,\gamma_t) dt+g(Y_T(\gamma), \mu_T)\right]\,.
\end{equation}
A mean-field game solution for $G_\infty$ is an admissible process $\hat \gamma\in\G$ which is optimal, i.e. $\hat\gamma\in\arg\min_{\gamma\in\G} J(\gamma)$, and satisfies the mean-field condition $\mu_t=\mathcal{L}(Y_t)$ for all $t\in[0,T]$. 
A mean-field solution $\hat\gamma$ of $G_\infty$ is said to be Markovian if $\hat\gamma_t=\hat\gamma(t,Y_{t-})$ where $\hat\gamma$
is a measurable function.\\

For the games to be well-defined and to find an approximate Nash equilibrium for the $n$ player game $G_n$ we have to require some integrability of the initial conditions of the state processes as well as some regularity on the functions
\begin{gather*}
b\colon[0,T]\times\R\times \cP(\R)\to\R , \quad \sigma\colon[0,T]\times\R\to\R ,\quad \beta\colon \cP(\R) \times A\to\R\,,\\
\lambda\colon[0,T]\to\R_+ , \quad f\colon[0,T]\times\R\times \cP(\R) \times A\to\R ,\quad g\colon\R\times  \cP(\R) \to\R\, ,
\end{gather*}
where $\R_+$ denotes the set of all positive real numbers.
\begin{assump}\label{main-ass}
\begin{enumerate}
\item \label{a:app:ic} The initial distribution $\chi$ belongs to $\cP^q(\R)$ for some $q>2$, $q \neq 4$.
\item  $b$ is a  Lipschitz function both in $x$ and $m$, $\sigma$ is Lipschitz in $x$, and $\lambda$ is Lipschitz in $\gamma$ and $\mu$. Namely,
there exist positive constants $L_b$, $L_\sigma$ and $L_\beta$  such that
\begin{gather*}
\abs{b(t,x,\mu)-b(t,y,\nu)}\le L_b\abs{x-y}+L_b d_{W,2}(\mu,\nu) \quad\forall x,y, \mu,\nu\in\cP^2(\R)\,,\,\forall t\in[0,T]\,,\\
\abs{\sigma(t,x)-\sigma(t,y)}\le L_\sigma\abs{x-y}\quad\forall x,y\in\R\,,\,\forall t\in[0,T]\,,\\
\abs{\beta(\mu,\gamma)-\beta(\nu,\eta)}\le L_\beta d_{W,2}(\mu,\nu) +L_\beta\abs{\gamma-\eta}\quad\forall \mu,\nu\in\cP^2(\R)\,, \forall\gamma,\eta\in A\,.
 \end{gather*} 
 Without loss of generality, we can assume $L_b=L_\sigma={L_\beta}=L$.\\
Moreover, $b$, $\sigma$, $\beta$ and $\lambda$ are bounded, i.e. there exists a positive constant $M$ satisfying
\[
\norm{b}_\infty+\norm{\sigma}_\infty+\norm{\beta}_\infty+\norm{\lambda}_\infty\le M\,.
\]
\item  $f$ and $g$ are Lipschitz functions in both $x$ and $m$, i.e.
there exist two positive constants $L_f$, $L_g$ such that
\begin{gather*}
\abs{f(t,x,\mu, \gamma)-f(t,y,\nu, \gamma)}\le L_f\abs{x-y}+L_f d_{W,2}(\mu,\nu) \quad\forall x,y,\mu,\nu\in\cP^2(\R)\,,\,\forall t\in[0,T]\,, \forall \gamma \in A\\
\abs{g(x,\mu)-g(y,\nu)}\le L_g\abs{x-y}+L_g d_{W,2}(\mu,\nu) \quad\forall x,y,\mu,\nu\in\cP^2(\R)\,.
 \end{gather*}
 As before, we can assume $L_f=L_g=L$.
\end{enumerate} 
\end{assump}
Here $d_{W,2}$ stands for the squared Wasserstein distance, while $\| \cdot \|_\infty$ denotes the sup-norm. From now on, to simplify the notation, we write $\cP(\R)$ for $\cP^2(\R)$ and $d_{W}$ for $d_{W,2}$.

Note that by standard results, Assumption~\ref{main-ass}(1-2) ensures that equations~\eqref{eq:XiN} and \eqref{eq:Y} admit a unique strong solution. The assumption $q \neq 4$ guarantees the applicability of \cite[Theorem 1]{fournier2015rate} to obtain the rate of convergence (see our Remark \ref{r:unifmY} for details). 

\subsection{$L^2$-estimates for the state processes and the empirical distribution process in $G_n$}

The two following lemmas provide estimates for the second moment of the process $X^n$ (as in \eqref{eq:XiN}) and for the corresponding empirical measure flow $\mu^n$. The proofs are rather standard, we include all the details for reader's convenience.

\begin{lemma}\label{lemma:supXiN}
Let Assumption \ref{main-ass} hold. Then, for each admissible strategy $\gamma\in\G^n$ the related controlled processes $X^{i,n}(\gamma)$ for $i=1,\dots,n$, solving \eqref{eq:XiN}, satisfy
\begin{equation}\label{def:tC}
\E \left[\sup_{t\in[0,T]} | X^{i,n}_t | ^2 \right] \le \hat C(\chi,T,\norm{b}_\infty,\norm{\sigma}_\infty,  \norm{\beta}_\infty ,\norm{\lambda}_\infty)\,,
\end{equation}
where the constant $\hat C$ is independent of  $n$ and $\gamma$.
\end{lemma}
\begin{proof}
There exists a constant $C$ such that
\begin{eqnarray*}
\abs{X^{i,n}_t}^2  & \le & C\abs{\xi_i}^2+C\abs{\int_0^t b(s,X^{i,n}_s,\mu^n_s) ds}^2 \\ && +C\abs{\int_0^t \sigma(s,X^{i,n}_s) dW^i_s}^2 +C\abs{\int_0^t \beta(\mu^n_{s-},\gamma^i_{s}) d\widetilde N^i_s}^2\,.
\end{eqnarray*}
Applying Jensen's and Burkholder-Davis-Gundy's inequalities, it follows that for a constant $C$ (which may change from line to line)
\begin{eqnarray*}
\E \left[\sup_{t\in[0,T]} \abs{X^{i,n}_t}^2 \right] &\le & C \E \left[\abs{\xi_i}^2 \right] + C t\E \left[\int_0^t \sup_{u\in[0,s]}\abs{b(s,X^{i,n}_s,\mu^n_s)}^2 ds \right]
\\&& + C  \E \left[ \int_0^t \sigma(s,X^{i,n}_s)^2 ds \right] + C \E \left[\int_0^t \beta (\mu^n_{s-},\gamma^i_{s})^2 \lambda(s) ds \right]\\
&\le & C(\chi) + C \norm{b}^2_\infty t^2 +C \norm{\sigma}_\infty^2 t + C \norm{\beta}^2 _\infty \norm{\lambda}_\infty t\\
&\le & \hat C(\chi,T,\norm{b}_\infty,\norm{\sigma}_\infty,\norm{\beta}_\infty,\norm{\lambda}_\infty),
\end{eqnarray*}
where we have used Assumption \ref{main-ass}(2) for the second inequality.
\end{proof}

\begin{lemma}\label{lemma:m2b}
Let Assumption \ref{main-ass} hold. Then, for each $n\ge 1$ and for each admissible strategy $\gamma\in\G^n$, there exists a constant $\hat C$ independent of $n$ and $\gamma$ such that
\begin{equation}
\label{eq:EsupmN2}
\E \left[ \sup_{t\in[0,T]} d_W(\mu^n_t,\delta_0)^2  \right] \le \hat C(\chi,T,\norm{b}_\infty,\norm{\sigma}_\infty,\norm{\beta}_\infty,\norm{\lambda}_\infty).
\end{equation}
\end{lemma}
\begin{proof}
We show that the constant $\hat C$ given in Lemma~\ref{lemma:supXiN} provides the required bound.
Indeed,
\begin{align*}
\E \left[ \sup_{t\in[0,T]} d_W(\mu^n_t,\delta_0)^2 \right] &\le \E \left[ \sup_{t\in[0,T]} \frac{1}{n}\sum_{i=1}^n \abs{X^{i,n}_t}^2  \right] \\
&\le  \frac{1}{n}\sum_{i=1}^n \E \left[ \sup_{t\in[0,T]}\abs{X^{i,n}_t}^2  \right] \le \hat C\,.\qedhere
\end{align*}
\end{proof}

\section{Markovian $\varepsilon$-Nash equilibrium}\label{nash}
This section presents the main result of this paper. Namely, we show how to construct approximate equilibria for the $n$-player game $G_n$ under the assumption that the mean-field game $G_\infty$ admits a Markovian solution. Therefore, consider the case when $\hat\gamma=\hat\gamma(t,Y_{t-})$ is a Markovian mean-field game solutions of $G_\infty$. Sufficient conditions ensuring the existence of Markovian mean-field game solution, together with an application to an illiquid interbank model, is given in our previous paper \cite{BCDP}.

Consider the game $G_n$, where each agent $i$ plays strategy $\hat\gamma=\hat\gamma(t,\hat X^{i,n}_{t-})$, i.e. each player follows the optimal strategy function $(t,x)\mapsto\hat\gamma(t,x)$ evaluated at the left-limit of his/her own state process $\hat X^{i,n}_{t-}$. 
The $n$-tuple $\hat X^n=(\hat X^{1,n},\dots,\hat X^{n,n})$ is defined as solution the following system
\begin{equation}\label{def:XiN}
\begin{aligned}
&d\hat X^{i,n}_t=b(t,\hat X^{i,n}_t,\mu^n_t) dt+\sigma(t,\hat X^{i,n}_t) dW^i_t +\beta(\mu
^n_{t-},\hat \gamma(t,\hat X_{t-}^{i,n})) d\widetilde N^i_t ,\\ &\hat X^{i,n}_0=\xi^i\,.
\end{aligned}
\end{equation}

For each player the strategy $\hat\gamma(t,\hat X^{i,n}_{t-})$ is admissible, i.e $(\hat\gamma(t,\hat X^{i,n}_{t-}))_{t\in[0,T]}\in\G$, being $\hat \gamma$ a (Borel)-measurable function by construction and $X^{i,n}_{t-}$ a predictable process as solution of the stochastic differential equation \eqref{def:XiN}.

All the results of this section are proved under the following standing assumption on the limiting mean-field game $G_\infty$:

\begin{assump}\label{markov-lip}
Assume that there exists a Markovian mean-field game solution $\hat\gamma_t =\hat\gamma(t,Y_{t-})$ for the game $G_\infty$, for some measurable function $\hat \gamma : [0,T]\times \mathbb R \to A$.
Moreover, the function $\hat\gamma (t,x)$ is Lipschitz continuous in the state variable $x$, i.e.
\begin{equation}\label{def:gL}
\abs{\hat \gamma(t,x)-\hat \gamma(t,y)}\le C_{\hat\gamma} \abs{x-y}\quad \forall t\in[0,T],\, \forall x,y\in\R\,,
\end{equation}
for some constant $C_{\hat\gamma}>0$.
\end{assump}

\begin{theorem}\label{thm:eNeM} 
Let Assumptions \ref{main-ass} and \ref{markov-lip} be fulfilled. If $\hat X^n$ is the solution of the system~\eqref{def:XiN}, the $n$-tuple $(\hat \gamma(t,\hat X^{1,n}_{t-}),\dots,\hat \gamma(t,\hat X^{n,n}_{t-}))$ is an $\varepsilon_n$-Nash equilibrium for the $n$-player game $G_n$, with $\varepsilon_n=O\left(n^{-\alpha/2}\right)\to 0$ as $n\to\infty$, where $\alpha=\min\left\{\frac{1}{2},\frac{q-2}{2}\right\}$.
\end{theorem}
Without loss of generality we can assume that $C_{\hat\gamma}=L$ as in Assumption~\ref{main-ass}.
\begin{notation}
From now on, the strategy profile $(\hat \gamma(t,\hat X^{1,n}_{t-}),\dots,\hat \gamma(t,\hat X^{n,n}_{t-}))$, $t \in [0,T]$, will be shortly denoted by $\hat\gamma^{\hat X^n}$. Observe that $\hat X$ is the solution of \eqref{eq:XiN} under such a strategy $\hat\gamma^{\hat X^n}$.
\end{notation}

It is worth noting that in the game $G_n$ all the players are symmetric in their behaviour. Indeed, all the SDEs defining their states (cf. equation~\eqref{eq:XiN}) and their payoff functions (cf. equation~\eqref{def:JiN}) have  the same form. For this reason in the following we will prove the main result Theorem \ref{thm:eNeM} considering without loss of generality deviations of Player $1$ only. Indeed the same arguments would apply to every other player in the game.

In the proof of Theorem~\ref{thm:eNeM}, we will focus on two different scenarios: the case when all the players choose to play according to the optimal recipe suggested by $G_\infty$, i.e. they all play $\hat \gamma(t,\hat X^{i,n}_{t-})$ as explained above, and the case when Player 1 deviates by choosing a different strategy $\eta\in\G$, i.e. $$(\eta ,\hat\gamma^{\hat X^n}_{-1})=((\eta_t , \hat\gamma(t,\hat X^{2,n}_{t-}),\dots,\hat\gamma(t,\hat X^{n,n}_{t-})))_{t \in [0,T]}\,.$$ 
\begin{notation}
In what follows, strategy $(\eta,\hat\gamma^{X^n}_{-1})$ will be simply denoted by $\eta^{\hat \gamma}$ and the solution of \eqref{eq:XiN} under this strategy will be denoted by $\widetilde X^{i,n}$.
\end{notation}

In the following we will also make use of the processes $Y^{i,n}$ ($i=1,\dots,n$) and $\widetilde Y^{1,n}$, given as solutions of
\begin{equation}\label{def:Y}
dY^{i,n}_t(\hat \gamma)=b(t,Y^{i,n}_t,\hat \mu_t) dt+\sigma(t,Y^{i,n}_t) dW^i_t+\beta(\hat \mu_{t-},\hat \gamma(t,Y^{i,n}_{t-})) d\widetilde N^i_t\,,\quad Y^{i,n}_0=\xi^i
\end{equation}
and of
\begin{equation}\label{def:tY}
d\widetilde Y^{1,n}_t(\eta)=b(t,\widetilde Y^{1,n}_t,\hat \mu_t) dt+\sigma(t,\widetilde Y^{1,n}_t) dW^1_t+\beta(\hat \mu_{t-},\eta_{t}) d\widetilde N^1_t\,,\quad \widetilde Y^{1,n}_0=\xi^1
\end{equation}
respectively, where $\hat \mu_t$ is the law of the state process of the limiting game $G_\infty$ under the Markovian MFG solution $\hat\gamma$. Note that since for each $i$, the process $Y^{i,n}$ is defined as the dynamics of a representative player in $G_\infty$, given in equation \eqref{eq:Y}, and $\hat\gamma$ is by construction a Markovian MFG solution, then $\mathcal{L}(Y^{i,n}_t)=\hat \mu_t$.

\begin{remark}
The definition of processes $Y^{1,n}$ and $\widetilde Y^{1,n}$ differs from the one of $\hat X^{i,n}$ and $\widetilde X^{i,n}$ due to the different \emph{measure} considered in the stochastic differential equations. Indeed in \eqref{eq:XiN}, the dynamics of $X^{i,n}$ is computed taking into account the associated empirical distribution of the system $X^n$, i.e. $\mu^n$ as defined in equation \eqref{def:empirical}, while the dynamics of $Y^{i,n}$ and $\widetilde Y^{1,n}$ in \eqref{def:Y} and \eqref{def:tY} are computed with respect to $\hat \mu$. This implies that $Y^{i,n}$ and $\widetilde Y^{1,n}$ do no longer depend on the other players' choices (and in the following we will say that they do not depend on $n$, for short). \end{remark} 

\begin{remark}\label{r:unifmY}
Consider the empirical distribution of the system $Y^n=(Y^{1,n},\dots,Y^{n,n})$ at time $t$, namely
\[
\mu^{Y,n}_t=\mu^{Y,n}_t(\hat\gamma)=\frac{1}{n}\sum_{i=1}^n \delta_{Y^{i,n}_t}\,.
\]
As first step we show that, in expectation, $\mu^{Y,n}$ converges towards $\hat \mu$ with respect to $d_W$ as $n\to\infty$, namely
\begin{equation}\label{eq:mYc}
\lim_{n\to\infty}\E\left[d_W(\hat\mu_t,\mu^{Y,n}_t)^2 \right] = 0 \,.
\end{equation}
Being $Y_t^{i,n}$ independent and identically distributed random variables with distribution $\hat \mu_t$, \cite[Theorem 1]{fournier2015rate} ensures that
\[
\E\left[d_W(\hat \mu_t,\hat \mu^{Y,n}_t)^2 \right]\le C(q) M_q^\frac{2}{q}(\hat\mu)\left(\frac{1}{\sqrt{n}}+\frac{1}{n^\frac{q-2}{q}}\right)
\]
where $C$ is a positive constant depending on $q$, and $M_q$ is defined as $M_q(\mu)=\int_\R \abs{x}^q\mu(dx)$.

Arguing as in the previous Lemma~\ref{lemma:supXiN} but exploiting the stronger hypothesis on the initial distribution $\chi\in\cP^q(\R)$ with $q>2$, one can prove that $Y$, solution to the SDE~\eqref{def:Y}, satisfies
\[
\E\left[\sup_{t\in[0,T]}\abs{Y_t}^q\right]\le \hat  C_2(\chi,T,\norm{b}_\infty,\norm{\sigma}_\infty,  \norm{\beta}_\infty,\norm{\lambda}_\infty)\,.
\]
This implies that $M_q(\hat \mu)$ is finite, and therefore
\begin{equation}\label{eq:O}
d_W(\hat\mu_t,\mu^{Y,n}_t)^2 = O \left( n^{-\alpha} \right)
\end{equation}
where $\alpha=\min \left\{\frac{1}{2},\frac{q-2}{q}\right\}$ and being $q>2$, it holds that
\begin{equation}\label{eq:mYc}
\lim_{n\to\infty}\E\left[d_W(\hat\mu_t,\mu^{Y,n}_t)^2 \right] = 0 
\end{equation}
uniformly in time.
\end{remark}

Now, we want to show that the process $Y^{i,n}(\hat\gamma)$  \emph{approximates} $\hat X^{i,n}$ as $n$ grows to infinity, in a sense that will be specified later. Note that being independent of $n$, the dynamics of $Y^{i,n}(\hat\gamma)$ is easier to study. In both the systems $\hat X^n$ and $Y^n$, all $n$ players choose the same strategy, or more precisely the same strategy form, i.e. $\hat\gamma(t,\hat X^{i,n}_{t-})$ and $\hat\gamma(t,Y^{i,n}_{t-})$. We stress once more that the dynamics in $\hat X$ depend on the actual law of the system, while the evolution of the state processes $Y^{i,n}$ depends on $\hat \mu$.

\begin{proposition}\label{prop:mNmh}
Let $\hat X^{i,n}$ and $Y^{i,n}$ be defined as in equation \eqref{def:XiN} and \eqref{def:Y}, respectively. Then we have
\begin{eqnarray}\label{eq:mX-mY=0}
\sup_{t\in[0,T]} \E\left[d_W (\mu^n_t ,  \hat \mu_t)^2 \right] =O \left( n^{-\alpha} \right)\,,\\
\label{eq:X-Y=0}
  \sup_{t\in[0,T]}  \E \left[ \abs{X^{i,n}_t-Y^{i,n}_t}^2 \right]= O \left( n^{-\alpha} \right)\,.
\end{eqnarray}
\end{proposition}
\begin{proof}
For each $t\in[0,T]$
\begin{multline*}
\abs{\hat X^{i,n}_t-Y^{i,n}_t}^2\le 3\left( \int_0^t \left(b(s,\hat X^{i,n}_s,\mu^n_s)-b(s,Y^{i,n}_s,\hat \mu_s)\right) ds\right)^2\\
+ 3\left(\int_0^t \left(\sigma(s,\hat X^{i,n}_s)-\sigma(s,Y^{i,n}_s)\right) dW^i_s\right)^2\\
+ 3\left(\int_0^t  \left( \beta(\mu_{s-}^n,\hat\gamma(s,\hat X^{i,n}_{s-}))-\beta(\hat \mu_{s-},\hat\gamma(s,Y^{i,n}_{s-}))\right) d\widetilde N^i_s\right)^2 \,.
\end{multline*}
Then, 
\begin{eqnarray*}
\E  \left[ \abs{\hat X^{i,n}_t-Y^{i,n}_t}^2\right] 
&\le & 3 t \E\left[  \int_0^t \abs{b(s,\hat X^{i,n}_s ,\mu^n_s)-b(s,Y^{i,n}_s,\hat \mu_s)}^2 ds\right]\\
&& + 3 \E\left[\int_0^t \abs{\sigma (s,\hat X^{i,n}_s) -\sigma(s,Y^{i,n}_s)}^2 ds\right]\\
&& + 3 \E\left[ \int_0^t \left( \beta(\mu_{s}^n,\hat\gamma(s,\hat X^{i,n}_{s}))-\beta(\hat \mu_{s},\hat\gamma(s,Y^{i,n}_{s}))\right)^2 \lambda(s)\, ds \right]\,.
\end{eqnarray*}
Using the Lipschitz continuity of functions $b$, $\sigma$ and $\beta$, given by Assumption \ref{main-ass}(2), and of $\hat\gamma(\cdot, x)$, as explained in equation \eqref{def:gL}, as well as the finiteness of $\E[\sup_{t\in [0,T]} d_W(\mu^n_t,\delta_0)^2 ]$, as in~\eqref{eq:EsupmN2}, we obtain
\begin{eqnarray}\label{eq:X-Y}
\E  \left[ \abs{\hat X^{i,n}_t-Y^{i,n}_t}^2\right] &\le &  6 t L^2 \int_0^t \left( \E  \left[\abs{\hat X^{i,n}_s -Y^{i,n}_s}^2\right]\ + \E \left[d_W(\mu^n_s,\hat\mu_s)^2 \right]\right) ds \nonumber  \\
&& + 3 L^2 \int_0^t  \E \left[d_W(\mu^n_s,\hat\mu_s)^2 \right] ds \nonumber \\
&& + 6  L^2\norm{\lambda}_\infty  \int_0^t  \left( \E \left[d_W(\mu^n_s,\hat\mu_s)^2 \right] + \E\left[\abs{\hat\gamma(s,\hat X^{i,n}_s)-\hat\gamma(s,Y^{i,n}_s)}^2\right] \right) ds \nonumber \\
&\le &  \widetilde C(T,L,M) \int_0^t \left( \E \left[ \abs{\hat X^{i,n}_s-Y^{i,n}_s}^2 \right] +  \E \left[d_W(\mu^n_s,\hat\mu_s)^2 \right] \right) ds ,
\end{eqnarray}
for a suitable constant $\widetilde C$. Moreover, by previous inequality \eqref{eq:X-Y}, we get
\begin{align*}
\E\left[d_W(\mu^n_t, \mu^{Y,n}_t)^2 \right]&\le   \frac{1}{n}\sum_{i=1}^n \E\left[ \abs{\hat X_t^{i,n}-Y^{i,n}_t}^2 \right]\\
&\le  \frac{\widetilde C }{n} \sum_{i=1}^n \int_0^t \left( \E\left[\abs{\hat X^{i,n}_s-Y^{i,n}_s}^2\right] +  \E \left[d_W(\mu^n_s,\hat\mu_s)^2 \right]\right) ds\,.
\end{align*}
Then, it holds that
\begin{equation}\label{eq:m-m}
\begin{aligned}
\E \left[d_W(\mu^n_t,\hat\mu_t)^2 \right]+ \frac{1}{n}\sum_{i=1}^n \E\left[ \abs{\hat X^{i,n}_t-Y^{i,n}_t} ^2\right] &\\
&\hspace{-5.5cm}\le 2 \E\left[ d_W(\mu^n_t, \mu^{Y,n}_t)^2 \right]  + 2 \E\left[ d_W (\hat\mu_t, \mu^{Y,n}_t)^2 \right] + \frac{1}{n}\sum_{i=1}^n \E\left[ \abs{\hat X^{i,n}_t-Y^{i,n}_t} ^2\right]\\
&\hspace{-5.5cm}\le 2 \E\left[ d_W (\hat \mu_t, \mu^{Y,n}_t)^2 \right]  + 2 \widetilde C\int_0^t \left( \E \left[d_W(\mu^n_s,\hat\mu_s)^2 \right]+ \frac{1}{n}\sum_{i=1}^n \E\left[ \abs{\hat X^{i,n}_s-Y^{i,n}_s} ^2\right]\right) ds\,.
\end{aligned}
\end{equation}
Therefore, by equation \eqref{eq:mYc} and Remark~\ref{r:unifmY}, we have 
\begin{equation*}\label{eq:m-m2}
\begin{aligned}
\E \left[d_W(\mu^n_t,\hat\mu_t)^2 \right] + \frac{1}{n}\sum_{i=1}^n \E\left[ \abs{\hat X^{i,n}_t-Y^{i,n}_t} ^2\right] &\\
&\hspace{-5.5cm}\le O \left( n^{-\alpha} \right) + 2 \widetilde C\int_0^t \left(\E \left[d_W(\mu^n_s,\hat\mu_s)^2 \right] + \frac{1}{n}\sum_{i=1}^n \E\left[ \abs{\hat X^{i,n}_s-Y^{i,n}_s} ^2\right] \right) ds ,
\end{aligned}
\end{equation*}
which implies equations~\eqref{eq:mX-mY=0} and \eqref{eq:X-Y=0}.
\end{proof}

In the previous estimates, we have considered the case when all the $n$ players are choosing the same strategy $\hat\gamma$. We now investigate what happen to the players' dynamics when Player 1 deviates from the  strategy profile $\hat\gamma^{\hat X^n}$ by playing some other strategy $\eta\in\G$. Note that in this case the dynamics of each player in $G_n$ is given by the solution to \eqref{eq:XiN} under the strategy $\eta^{\hat \gamma}$, i.e. $\widetilde X^{i,n}$.

\begin{proposition}\label{prop:mtmN}
Let $\hat X$ and $\widetilde X$ be the solutions of the system \eqref{eq:XiN}, when the strategy profile is given by $\hat\gamma^{\hat X^n}$ and by $\eta^{\hat\gamma}$, respectively. We denote by $\mu^n$ and $\widetilde \mu^n$ the empirical distribution of the two systems. Then, 
\[ 
\sup_{t\in[0,T]} \E \left[ d_W (\mu^n_t,\widetilde \mu^n_t)^2 \right]=  O \left( n^{-1} \right)\,.
\]
Moreover, considering $\widetilde Y^{1,n}$ defined in equation \eqref{def:tY}, it holds that
\begin{equation}\label{eq:Xt-Yt}
 \sup_{t\in[0,T], \,\eta\in \G}\E\left[\abs{\widetilde X_t^{1,n}-\widetilde Y^{1,n}_t}^2\right] = O \left( n^{-\alpha} \right)\,.
\end{equation}
\end{proposition}
\begin{proof}
Firstly, we consider Player 1. By Lemma~\ref{lemma:supXiN}, the Lipschitz condition on $b$, $\sigma$ and $\beta$, and the boundedness of function $\lambda$ given by Assumption~\ref{main-ass}(2) imply
\begin{align*}
\E\left[\abs{\hat X_t^{1,n}-\widetilde X^{1,n}_t}^2\right] &\le  3 t L^2\int_0^t \E\left[\abs{\hat X_s^{1,n}-\widetilde X^{1,n}_s}^2\right]+ \E \left[ d_W (\mu^n_s,\widetilde \mu^n_s)^2 \right]ds 
+ 3 L^2\int_0^t\E\left[\abs{\hat X_s^{1,n}-\widetilde X^{1,n}_s}^2\right]ds\\
&\hspace{1cm}	
+3\E\left[\int_0^t \left( L \, d_W (\mu^n_s,\widetilde \mu^n_s) +L \abs{\hat\gamma(s,\hat X^{1,n}_s)-\eta}     \right)^2\lambda(s)\,ds \right]
\\
%
&\le 12 L^2 \hat C T(2 T +   2M+1)+12 L^2 A^2_\infty M   T  =: C_1 ,
\end{align*}
where the constant $\hat C$, given as in equation \eqref{def:tC}, is independent of $n$ and then so is $C_1$. Furthermore, by definition $C_1$ does not depend on $\eta$ either.

On the other hand, the other players for $i=2,\dots,n$ play the strategy $\hat\gamma(t,X^{i,n}_{t-})$ in both cases, then, to find an estimate for $\E [ | \hat X_t^{i,n}-\widetilde X^{i,n}_t |^2 ]$ we can argue as in \eqref{eq:X-Y}. Finally, following the same idea as to obtain \eqref{eq:m-m}, we have that
\begin{align*}
\E \left[ d_W (\mu^n_t,\widetilde \mu^n_t)^2 \right] &\le\frac{1}{n}\E\left[\abs{\hat X^{1,n}_t-\widetilde X^{1,n}_t}^2\right]+\frac{1}{n}\sum_{i=2}^n\E\left[\abs{\hat X^{i,n}_t-\widetilde X^{i,n}_t}^2\right]\\
&\le \frac{C_1}{n}+\frac{\widetilde C}{n}\sum_{i=2}^n \int_0^t \left( \E\left[\abs{\hat X^{1,n}_s-\widetilde X^{1,n}_s}^2\right] + \E \left[ d_W (\mu^n_s,\widetilde \mu^n_s)^2 \right]  \right) \,ds
\end{align*}
and therefore
\begin{multline*}
\E \left[ d_W (\mu^n_t,\widetilde \mu^n_t)^2 \right] +\frac{1}{n}\sum_{i=2}^n\E\left[\abs{\hat X^{i,n}_t-\widetilde X^{i,n}_t}^2\right]\\
\le \frac{C_1}{n}+\frac{2 \widetilde C}{n}\sum_{i=2}^n \int_0^t \left( \E\left[\abs{\hat X^{1,n}_s-\widetilde X^{1,n}_s}^2\right] +   \E \left[ d_W (\mu^n_s,\widetilde \mu^n_s)^2 \right]  \right) \,ds\,.
\end{multline*}
Finally, applying again Gronwall's lemma, it is found that
\begin{equation}\label{eq:K1}
 \E \left[ d_W (\mu^n_t,\widetilde \mu^n_t)^2 \right] +\frac{1}{n}\sum_{i=2}^n \E\left[\abs{\hat X^{i,n}_t-\widetilde X^{i,n}_t}^2\right]
\le \frac{1}{n} K_1(T,L,A_\infty, C_1)\,,
\end{equation}
with $K_1$ independent of $n$, $t$ and $\eta$. Therefore 
\[\sup_{t\in[0,T]\,\eta\in\G}  \E \left[ d_W (\mu^n_t,\widetilde \mu^n_t)^2 \right] =   O \left( n^{-1} \right) \,.\]
Lastly, arguing as in the proof of Proposition~\ref{prop:mNmh}, by considering $\widetilde Y^{1,n}$ as defined in \eqref{def:tY} we have
\begin{equation*}\label{eq:tK}
\begin{aligned}\E\left[\abs{\widetilde X_t^{1,n}-\widetilde Y^{1,n}_t}^2\right] &\le
3 (2t +1)L^2 \int_0^t \E  \left[\abs{\hat X^{i,n}_s -Y^{i,n}_s}^2\right]\,ds  \\
&\hspace{2cm}+ 3 L^2(2t+  \norm{\lambda}_\infty) \int_0^t  \E \left[ d_W (\mu^n_s,\hat \mu_s)^2 \right] \,ds
\\& \le \widetilde K(T,L, M)\int_0^t \left( \E\left[\abs{\widetilde X_s^{1,n}-\widetilde Y^{1,n}_s}^2\right] +  \E \left[ d_W (\mu^n_s,\hat \mu_s)^2 \right] \right) \,ds ,
\end{aligned}
\end{equation*}
so that 
\[
\E\left[\abs{\widetilde X_t^{1,n}-\widetilde Y^{1,n}_t}^2\right] \le  \widetilde K(T,L,M)\int_0^t \E\left[\abs{\widetilde X_s^{1,n}-\widetilde Y^{1,n}_s}^2\right]\,ds +\widetilde K(T,L,M)  O \left( n^{-\alpha} \right)\,.
\]
Hence Gronwall's lemma implies
\begin{equation}\label{eq:Xt-Yt2}
\E\left[\abs{\widetilde X_t^{1,n}-\widetilde Y^{1,n}_t}^2\right] \le \bar K(T, L,M) \, O \left( n^{-\alpha} \right) ,
\end{equation}
for a suitable constant $\bar K$ independent of $n$ and $t$, and therefore
\begin{equation*}
\sup_{t\in[0,T],\,\eta\in \G}\E\left[\abs{\widetilde X_t^{1,n}-\widetilde Y^{1,n}_t}^2\right] =  O \left( n^{-\alpha} \right)\,.
\end{equation*}
\end{proof}

\begin{remark}
It is crucial here and in what follows that the constants $K_1$ and $\bar K$ appearing in \eqref{eq:K1} and \eqref{eq:Xt-Yt2} do not depend on how Player 1 deviates from the  strategy profile $\hat\gamma^{\hat X^n}$.
\end{remark}

In order to prove Theorem~\ref{thm:eNeM}, we will make use of the following two operators: $\widetilde J_n \colon \G^n\to \R$ and $\widetilde J \colon \G\to \R$, defined by
\begin{equation}\label{def:tJN}
\widetilde  J_n (\gamma)=\E\left[\int_0^T f(t,X^{1,n}_t(\gamma),\hat \mu_t,\gamma^1_t)\,dt+g(X^{1,n}_T(\gamma), \hat \mu_T)\right]
\end{equation}
and
\begin{equation}\label{def:tJ}
\widetilde J (\eta) = \E\left[\int_0^T f(t,Y^{1,n}_t,\hat \mu_t,\eta_t)\,dt+g(Y^{1,n}_T, \hat \mu_T)\right]
\end{equation}
respectively, where $X^{1,n}(\gamma)$ and $Y^{1,n}(\eta)$ are given as in \eqref{eq:XiN} and \eqref{def:tY}.
It is worth observing that $\widetilde J$ does not depend on the number of players in the game $n$. Indeed, $Y^{1,n}$ follows the dynamics of a representative player in the mean-field game  $G_\infty$, and therefore, $\widetilde J$ is exactly the expected cost of the strategy $\eta$ in $G_\infty$ with respect the flow of measures $\hat \mu$, as given in equation~\eqref{def:Jinf}. Therefore, since $\hat\gamma(t,Y^{1,n}_t)$ is by construction one of the minimizing strategies, i.e. $\hat\gamma\in\arg\min_{\gamma\in\G} J(\gamma)$, we have that 
\begin{equation}\label{eq:opt}
\widetilde J (\hat \gamma(t,Y^{1,n}_{t-})) \le \widetilde J(\eta)\quad\text{for all $\eta \in \G$ and $t\in [0,T]$}\,.
\end{equation}

As first step, we show that the value of Player $1$ in the game $G_n$, when he/she deviates from the candidate Nash equilibrium $\hat\gamma^n$ to a different admissible strategy $\eta\in\G$, that is $J^{1,n}(\eta^{\hat\gamma})$ given in equation \eqref{def:JiN}, can be approximated (when $n$ is large) with $\widetilde J_n (\eta^{\hat\gamma})$, that is the expected cost computed under the same strategy profile $\eta^{\hat \gamma}$, but evaluated with respect to the mean measure $\hat m$.
\begin{proposition}\label{prop:app1}
Let $(t,x)\mapsto \hat\gamma(t,x)$ be as in Assumption \ref{markov-lip}. Consider the strategy profile 
\[ \hat\gamma^{\hat X^n}_t =(\hat\gamma(t,\hat X^{1,n}_{t-}),\dots,\hat\gamma(t,\hat X^{n,n}_{t-})), \quad t \in [0,T],\] and let $\eta$ be an admissible strategy in $\G$. Then
\begin{equation}\label{eq:app1}
\sup_{\eta\in\G} \abs{J^{1,n}(\eta^{\hat\gamma})-\widetilde  J_n(\eta^{\hat\gamma})} = O \left( n^{-\frac{\alpha}{2}} \right)\,.
\end{equation}
\end{proposition}
\begin{proof}
By definition \eqref{def:tJN}-\eqref{def:tJ} and Assumption~\ref{main-ass}(3), the distance between the two operators $J^{1,n}$ and $\widetilde  J_n$ can be bounded as follows.
\begin{align*}
\abs{J^{1,n}(\eta^{\hat\gamma})-\widetilde  J_n(\eta^{\hat\gamma})}&\le \E\left[ \int_0^T\abs{f(t,\widetilde X^{1,n}_t,\widetilde \mu^n_t,\eta_t)-f(t,\widetilde X^{1,n}_t,\hat \mu_t,\eta_t)}\,dt\right] \\
&\hspace{2cm}+\E\left[\abs{g(\widetilde X^{1,n}_T,\widetilde \mu^n_T)-g(\widetilde X^{1,n}_T,\hat \mu_T)}\right]\\
& \le L\int_0^T\E\left[d_W(\widetilde \mu^n _t, \hat \mu_t)\right] \,dt+L\E\left[d_W(\widetilde \mu^n_T,\hat \mu_T)\right]\,.
\end{align*}
Hence equation~\eqref{eq:app1} follows from previous results in Proposition~\ref{prop:mNmh} and Proposition~\ref{prop:mtmN}, since
\[
\E\left[d_W(\widetilde \mu^n _t, \hat \mu_t)\right]\le\left(\E\left[d_W(\widetilde \mu^n _t, \hat \mu_t)^2 \right]\right)^\frac{1}{2}=O \left( n^{-\frac{\alpha}{2}}\right)\,.
\]
\end{proof}

As second step, we approximate $\widetilde J_n(\eta^{\hat \gamma})$ with $\widetilde J(\eta)$, that is the expected cost for playing $\eta$ in the mean-field game $G_\infty$.
\begin{proposition}\label{prop:app2}
Let $(t,x)\mapsto \hat\gamma(t,x)$ represent the Markovian structure of a mean-field game solution of the game $G_\infty$, $\hat\gamma^{\hat X^n}_t =(\hat\gamma(t,\hat X^{1,n}_{t-}),\dots,\hat\gamma(t,\hat X^{n,n}_{t-}))$, for $t \in [0,T]$, and let $\eta\in\G$ be an admissible strategy. Then
\begin{equation}\label{eq:app2}
\sup_{\eta\in\G}\abs{\widetilde J_n (\eta^{\hat\gamma})-\widetilde J(\eta)} = O \left( n^{-\frac{\alpha}{2}} \right)\,.
\end{equation}
\end{proposition}
\begin{proof}
Arguing as in Proposition~\ref{prop:app1}, we have that
\begin{align*}
\abs{\widetilde J_n (\eta^{\hat\gamma})-\widetilde J(\eta)}&\le\E\left[\int_0^T\abs{f(t, \widetilde X^{1,n}_t,\hat \mu_t,\eta_t)-f(t,\widetilde Y^{1,n}_t,\hat \mu_t,\eta_t)}\,dt\right] \\
&\hspace{2cm}+\E\left[\abs{g(\widetilde X^{1,n}_T,\hat \mu_T)-g(\widetilde Y^{1,n}_T,\hat \mu_T)}\right]\\
&\le L \int_0^T\E\left[\abs{\widetilde X^{1,n}_t -\widetilde Y^{1,n}_t}\right] \,dt+L\E\left[\abs{\widetilde X^{1,n}_T-\widetilde Y^{1,n}_T}\right]\,.
\end{align*}
Since by Proposition~\ref{prop:mtmN}, $\E\left[ \abs{ \widetilde X^{1,n}_t(\eta^{\hat\gamma})- \widetilde Y^{1,n}_t(\eta)}\right]= O \left( n^{-\frac{\alpha}{2}} \right)$, then
\begin{equation}\label{eq:app2}
 \sup_{\eta\in\G} \abs{\widetilde J_n (\eta^{\hat\gamma})-\widetilde J(\eta)} = O \left( n^{-\frac{\alpha}{2}} \right),
\end{equation}
as claimed.
\end{proof}

\begin{proof}[Proof of Theorem~\ref{thm:eNeM}]
Given an admissible strategy $\eta \in \G$, let 
\[ \varepsilon^1_n =4 \sup_{\eta\in\G}\abs{J^{1,n}(\eta^{\hat\gamma})-\widetilde  J_n (\eta^{\hat\gamma})}, \quad \varepsilon^2_n =4\sup_{\eta\in\G}\abs{\widetilde J_n (\eta^{\hat\gamma})-\widetilde J(\eta)}, \quad \varepsilon_n =\varepsilon^1_n +\varepsilon^2_n.\]
Then
\begin{equation*}
J^{1,n}(\eta^{\hat\gamma})\ge -\frac{\varepsilon_n }{2}+ \widetilde J(\eta)\ge -\frac{\varepsilon_n}{2} + \widetilde J(\hat\gamma)\ge -\varepsilon_n + J^{1,n}(\hat\gamma),
\end{equation*}
which gives \eqref{eq:eNe} for Player 1. More in detail, the first and the third inequalities are guaranteed by Proposition~\ref{prop:app1} and Proposition~\ref{prop:app2} respectively, whereas the second inequality is justified in equation~\eqref{eq:opt}. The symmetry of the game $G_n$ guarantees that $(\hat\gamma(t,X^{1,n}_{t-}),\dots,\hat\gamma(t,X^{n,n}_{t-}))$, for $t \in [0,T]$, is an $\varepsilon$-Nash equilibrium of it.

The rate of convergence, i.e. $\varepsilon_n=  O \left( n^{-\frac{\alpha}{2}} \right)$, is also granted by the previous approximations in Propositions~\ref{prop:app1} and \ref{prop:app2}.
\end{proof}

\section*{Acknowledgements}
Professors L. Campi and L. Di Persio would like to thank the FBK (Fondazione Bruno Kessler) - CIRM (Centro Internazionale per la Ricerca Matematica) for having funded their ``Research in Pairs'' project \emph{McKean-Vlasov dynamics with L\'evy noise with applications to systemic risk},
during the period from November 1-8, 2015 , to which Chiara Benazzoli has actively participated.\\
Prof. Luca Di Persio and Chiara Benazzoli would like to thank the Gruppo Nazionale per l'Analisi Matematica, la Probabilità e le loro Applicazioni (GNAMPA) for having funded the projects \emph{Stochastic Partial Differential Equations and Stochastic Optimal Transport with Applications to Mathematical Finance}, coordinated by L. Di Persio (2016), and \emph{Metodi di controllo ottimo stocastico per l’analisi di problem debt management}, coordinated by Dr. A. Marigonda (2017), since both of which have contributed to sustain the development of the present paper.

\printbibliography
\end{document}